\documentclass[12pt]{amsart}
\usepackage{amsthm,amssymb,hyperref}
\usepackage[all]{xy}
\newcommand{\Q}{\mathbb{Q}}
\newcommand{\Z}{\mathbb{Z}}
\newcommand{\F}{\mathbb{F}}
\newcommand{\R}{\mathbb{R}}

\renewcommand{\P}{\mathbb{P}}

\usepackage{amsthm}
\newtheorem{thm}{Theorem}

\newtheorem*{ack}{Acknowledgements}
\newtheorem*{rem}{Remark}
\newcommand{\Aut}{{\rm Aut}}

\newcommand{\im}{{\rm im}~}
\DeclareFontFamily{U}{wncy}{}
\DeclareFontShape{U}{wncy}{m}{n}{<->wncyr10}{}
\DeclareSymbolFont{mcy}{U}{wncy}{m}{n}
\DeclareMathSymbol{\Sha}{\mathord}{mcy}{"58}

\usepackage{color}
\definecolor{acolor}{rgb}{0.6,0.4,0.8}

\definecolor{jcolor}{rgb}{0.2,0.8,0.2}

\begin{document}

\title[Sums of two sixth powers]{Integers that are sums of two rational sixth powers}
\author{Alexis Newton}
\author{Jeremy Rouse}
\subjclass[2010]{Primary 11G05; Secondary 14H45, 11Y50}
\begin{abstract}
  We prove that $164634913$ is the smallest positive integer
  that is a sum of two rational sixth powers but not a sum of two integer
  sixth powers. If $C_{k}$ is the curve $x^{6} + y^{6} = k$, we use the existence
  of morphisms from $C_{k}$ to elliptic curves, together with
  the Mordell-Weil sieve, to rule out the existence of rational points on $C_{k}$ for various $k$.
\end{abstract}

\maketitle

\section{Introduction and Statement of Results}
\label{intro}

Fermat's classification of which integers are the sum of two integer squares
allows one to prove that if $k$ is a positive integer and there are $a, b \in \Q$ with $a^{2} + b^{2} = k$, then there are $c, d \in \Z$ with $c^{2} + d^{2} = k$.
(For more detail, see Proposition 5.4.9 of \cite{Cohen}.)

However when considering higher powers, the analogous result is no longer true.
In particular, $6 = \left(\frac{17}{21}\right)^{3} + \left(\frac{37}{21}\right)^{3}$ despite the fact that there are no integers $x$ and $y$ so that $x^{3} + y^{3} = 6$. In \cite{BremnerMorton}, Andrew Bremner and Patrick Morton prove that $5906 = \left(\frac{25}{17}\right)^{4} + \left(\frac{149}{17}\right)^{4}$ is the smallest positive integer which is a sum of two rational fourth powers, but
not a sum of two integer fourth powers. Their proof involves a number of
explicit calculations involving class numbers and units in rings of integers
of number fields.

It is natural to ask what can be said about values of $n > 4$. In particular,
is there always an integer $k$ that is a sum of two rational $n$th powers
but not a sum of two integer $n$th powers? In John Byrum's unpublished
undergraduate thesis (conducted under the direction of the second author),
he proves that if there is a prime $p \equiv 1 \pmod{2n}$ with $p \leq 2n^{2} - n + 1$, then there is a positive integer $k$ that is a sum of two rational $n$th
powers but not a sum of two integer $n$th powers. It is not known that
one can find such a prime $p$. Even assuming the generalized Riemann hypothesis (GRH), the strongest known result
at this time is that the smallest prime $p \equiv 1 \pmod{2n}$ is less than or equal to
$(\phi(2n) \log(2n))^{2}$ (by Corollary 1.2 of \cite{LamzouriLiSound}),
which is not sufficiently small unless $n = 3$. It is conjectured that
the smallest prime $p \equiv a \pmod{q}$ satisfies $p \ll q^{1+\epsilon}$, which would be sufficient.

The goal of the present paper is to handle the case $n = 6$ and prove an
analogous result to that of Bremner and Morton. Our main result
is the following.
\begin{thm}
\label{main}
The smallest positive integer which is a sum of two rational sixth powers but not a sum of two integer sixth powers is
\[
164634913 = \left(\frac{44}{5}\right)^{6} + \left(\frac{117}{5}\right)^{6}.
\]
\end{thm}

To prove the main result we must show that if an integer $k <
164634913$ is sixth-power free and is not a sum of two integer sixth
powers, then it is not a sum of two rational sixth powers either. We
proceed by studying when $C_{k} : x^{6} + y^{6} = kz^{6}$ has a
solution in $\Q_{p}$ for all primes $p$, which reduces the number of
necessary $k$ to consider to $111625$. To handle these, we decompose
the Jacobian of $C_{k}$ (up to isogeny) as a product of ten elliptic
curves elliptic curves, each with $j$-invariant zero. We use a
combination of techniques to show that for each of the remaining
$111625$ values of $k$, either one of these elliptic curves has rank
zero, or we determine a finite-index subgroup of the Mordell-Weil
group and use the Mordell-Weil sieve. Code and output files verifying
our computations are available on GitHub \href{https://github.com/newtan18/Sums-of-Two-Sixth-Powers/}{here}.

We note that there are infinitely many integers that are sums of two rational
sixth powers but not sums of two integer sixth powers.
\begin{thm}
\label{family} Let $t$ be an integer and $f_{1} = (2863 + 10764t)/13$ and $f_{2} = (1207 + 26455t)/13$. Then $f_{1}^{6} + f_{2}^{6}$ is an integer that is a sum
of two rational sixth powers, but not a sum of two integer sixth powers.
\end{thm}

The polynomial $f_{1}^{6} + f_{2}^{6}$ is constructed so that the coefficients of $t, t^2, \ldots, t^6$ are all multiples of $13$, while the constant coefficient is equivalent to $5 \pmod{13}.$ Since it is impossible to have an integer equivalent to $5 \pmod{13}$ be a sum of two integer sixth powers, we have our result.

\begin{rem}
It seems likely that no positive integer can be written as a sum of
two rational sixth powers in more than one way. In \cite{Ekl}, Randy
Ekl searched for integer solutions to $a^{6} + b^{6} = c^{6} + d^{6}$
with $a \ne c$ and $a \ne d$ and found none for which $a^{6} + b^{6} <
7.25 \cdot 10^{24}$. The surface $X : a^{6} + b^{6} = c^{6} + d^{6}$
is a surface of general type, and the Bombieri-Lang conjecture predicts
that there are only finitely many rational points on $X$ that do not
lie on a genus $0$ or $1$ curve.
\end{rem}

\begin{ack}
This work represents joint work done when the first author was a
master's student at Wake Forest University. Computations were done in
Magma \cite{Magma} version 2.26-9 on a desktop with an Intel i9-11900K
CPU and 128 GB of RAM. The authors thank the referees for the
detailed suggestions including the decomposition of the Jacobian of $x^{6}+y^{6}=kz^{6}$ and computational suggestions that led to an unconditional main result.
\end{ack}

\section{Background}
\label{back}

We let $\Q_{p}$ denote the field of $p$-adic numbers. A necessary
condition for a curve $C/\Q$ to have a rational point is for $C(\R)
\ne \emptyset$ and $C(\Q_{p}) \ne \emptyset$ for all primes $p$. If
$C$ satisfies this condition, we say that $C$ is locally solvable.

For our purposes, an elliptic curve is a smooth cubic curve of the form
\[
E : y^{2} + a_{1} xy + a_{3} y = x^{3} + a_{2} x^{2} + a_{4} x + a_{6}.
\]
There is a natural abelian group structure on $E(\Q)$, the set of rational
points on $E$.

\begin{thm}[\cite{Silverman}, Theorem VIII.4.1]
  The group $E(\Q)$ is finitely generated. That is, there is a finite group
  $E(\Q)_{\rm tors}$ so that $E(\Q) \cong E(\Q)_{\rm tors} \times \Z^{r}$ for some
  non-negative integer $r$.
\end{thm}
The non-negative integer $r$ is called the rank of $E(\Q)$. The Birch
and Swinnerton-Dyer conjecture predicts that if $L(E,s)$ is the $L$-function
of $E$, the ${\rm ord}_{s=1} L(E,s) = r$. This is proven
in the case that $r = 0$ or $1$ by Gross-Zagier \cite{GrossZagier} and
Kolyvagin \cite{Kolyvagin}. 

For $k \ne 0$, the curve $C_{k} : x^{6} + y^{6} = kz^{6}$ is a curve
of genus $10$. For $k = 1$, the decomposition of the Jacobian is
worked out in \cite{Aoki} and it follows that each factor of
$J(C_{1})$ is an elliptic curve with $j$-invariant zero. We will show
in Section~\ref{maps} that there are non-constant morphisms from $C_{k}$ to six different elliptic curves of the form $E_{a} : y^{2} = x^{3} + a$. The torsion
subgroup of an elliptic curve of the form $E_{a}$ has been known for some time.

\begin{thm}[\cite{Fueter}]
\label{torsion}
If $E_{a} : y^{2} = x^{3} + a$, then
\[
E_{a}(\Q)_{\rm tors} \cong \begin{cases}
  \Z/6\Z & \text{ if } a \text{ is a sixth power, } \\
  \Z/3\Z & \text{ if } a \text{ is a square but not a sixth power }\\
  \Z/3\Z & \text{ if } a \text{ is } -432 \text{ times a sixth power, }\\
  \Z/2\Z & \text{ if } a \text{ is a cube but not a sixth power, }\\
  \Z/1\Z & \text{ otherwise. }
\end{cases}
\]
\end{thm}
There is a torsion point on $y^{2} = x^{3} + a$ for which $x$ and $y$ are
both nonzero only when $a = -432k^{6}$ (namely $(12k^{2} : \pm 36k^{3} : 1)$) or $a = k^{6}$ (namely $(2k^{2} : \pm 3k^{3} : 1)$).

The Mordell-Weil sieve is a technique for proving that a curve $C$ has
no rational points. For a thorough treatment of this subject, see the
paper of Nils Bruin and Michael Stoll \cite{BruinStoll}.

Let $J$ be the Jacobian of $C$ and assume that we have in hand
a $\Q$-rational divisor $D$ of degree $1$ on $C$. Let $\iota \colon C \to J$
be the map $\iota(P) = P-D$. Fix a positive integer $N$ and a finite set $S$
of primes. We then have the following commutative diagram.
\[
\xymatrix{
C(\Q) \ar[rr]^{\iota} \ar[d] & & J(\Q)/NJ(\Q) \ar[d]^{\alpha}\\
\prod_{p \in S} C(\F_{p}) \ar[rr]^{\beta} & & \prod_{p \in S} J(\F_{p})/NJ(\F_{p})}
\]
If $C(\Q)$ is non-empty, then there will be an element in
$\prod_{p \in S} J(\F_{p})/NJ(\F_{p})$ that is in the image of both $\alpha$
and $\beta$. Therefore, if we can find an $N$ and a finite set $S$ for which
the image of $\alpha$ and the image of $\beta$ are disjoint, then $C(\Q)$
is empty.

The curve $C_{k}$ has maps to six different elliptic curves: $E_{k}$,
$E_{4k}$, $E_{-k^{2}}$, $E_{16k^{2}}$, $E_{k^{3}}$ and $E_{-4k^{4}}$. As a consequence,
we will replace $J$ with one of these six curves in our applications. Computing
the Mordell-Weil group (or a finite index subgroup thereof) for one of these
six elliptic curves allows us to apply the Mordell-Weil sieve to $C_{k}$.

\section{Finding an integer that is a sum of two rational sixth powers}

We will describe briefly how the representation of $164634913 = (44/5)^{6} + (117/5)^{6}$ was generated by the authors. We seek integers $x$, $y$ and $m$
for which $x^{6} + y^{6} \equiv 0 \pmod{m^{6}}$ with $\gcd(x,m) = \gcd(y,m) = 1$.
This equation implies that $xy^{-1}$ must have order $4$ or $12$ in $(\Z/m^{6} \Z)^{\times}$ which implies that all the prime factors of $m$ must be $\equiv 1 \pmod{4}$. The smallest such $m$ is $m = 5$.

We let $q = 1068$ be an element of order $4$ in $(\Z/5^{6} \Z)^{\times}$. Then $1^{6} + q^{6} \equiv 0 \pmod{5^{6}}$. We wish to find an integer $a$ so that
$\pm a \bmod 5^{6}$ and $\pm aq \bmod 5^{6}$ are both small. We consider the
lattice $L \subseteq \R^{2}$ consisting of all vectors $\left\{ \begin{bmatrix} x \\ y \end{bmatrix} : y \equiv qx \pmod{5^{6}} \right\}$. We find that
an $LLL$-reduced basis for this lattice consists of $\begin{bmatrix} 117 \\ 44 \end{bmatrix}$ and $\begin{bmatrix} 44 \\ -117 \end{bmatrix}$ from which
we obtain $164634913 = \left(\frac{44}{5}\right)^{6} + \left(\frac{117}{5}\right)^{6}$.

We wish to note that this representation was found at least twice previously.
First, it is given by J. M. Gandhi on page 1001 of \cite{Gandhi}. Second,
it was noted by John W. Layman on Oct. 20, 2005 in connection with OEIS sequence A111152 (the smallest integers that are a sum of two rational $n$th powers but not a sum of two integer $n$th powers).

For integers of the form $x^{n} + y^{n}$ with $n$ odd, there are no local restrictions, and setting $x = \frac{2^{n-1} - 1}{2}$ and $y = \frac{2^{n-1} + 1}{2}$ leads to a fairly small integer that is a sum of two rational $n$th powers.
For $n = 5$, this leads to $68101 = \left(\frac{15}{2}\right)^{5} + \left(\frac{17}{2}\right)^{5}$.

\section{Local solvability}

In this section we study the question of when $C_{k} : x^{6} + y^{6} = k$ is
locally solvable. 

\begin{thm}
Let $k$ be a positive integer which is sixth power free. Then $C_{k}$ is locally solvable if and only if $C_{k}$ has points over $\Q_{p}$ for all primes $p < 400$ and all odd prime factors $p \mid k$ have $p \equiv 1 \pmod{4}$.
\end{thm}
\begin{proof}
  The curve $C_{k}$ is smooth over $\F_{p}$ for all primes $p$ other than $2$, $3$ and those dividing $k$. Since $C_{k}$ has genus $10$, Hasse's theorem gives
  that $|C_{k}(\F_{p})| > p+1 - 2 \cdot 10 \sqrt{p}$ provided $C_{k}/\F_{p}$ is smooth.
  The latter quantity is positive if $p > 400$. Also, Hensel's lemma implies
  that if $C_{k}(\F_{p})$ has a non-singular point, then it lifts to a non-singular point of $C_{k}(\Q_{p})$ and hence $C_{k}(\Q_{p}) \ne \emptyset$.

  If $p \mid k$ and $p \equiv 1 \pmod{4}$, then $p = a^{2} + b^{2}$ for some $a, b \in \Z$. Since $a^{2} + b^{2} \mid a^{6} + b^{6}$, we have that $(a : b : 1)$ is a smooth point on $C_{k}/\F_{p}$ and therefore $C_{k}(\Q_{p}) \ne \emptyset$.

  Suppose that $p \mid k$, $p \equiv 3 \pmod{4}$ and $(x_{0} : y_{0} :
  z_{0}) \in C_{k}(\Q_{p})$ with $x_{0}, y_{0}, z_{0} \in \Z_{p}$, not all of which are multiples of $p$. It follows that $p$ divides one of $x_{0}^{2} + y_{0}^{2}$
  or $x_{0}^{4} - x_{0}^{2} y_{0}^{2} + y_{0}^{4}$, and both of these imply that
  $x_{0} \equiv y_{0} \equiv 0 \pmod{p}$. It follows that
  $x_{0}^{6} + y_{0}^{6} = kz_{0}^{6}$ is a multiple of $p^{6}$. Since $k$ is sixth power free, it follows that $p \mid z_{0}$, which is a contradiction. Thus, $C_{k}(\Q_{p}) = \emptyset$.
\end{proof}

We note that the smallest positive integer $k$ that is not a sum of two
integer sixth powers for which $C_{k}$ is locally solvable is $k = 2017$.

To enumerate the $k < 164634913$ for which $C_{k}$ is locally solvable,
we note that if $C_{k}$ is locally solvable, then $k \equiv 1, 2 \pmod{7}$
and $k \equiv 1, 2 \pmod{8}$ and $k \equiv 1, 2 \pmod{9}$. Also, for each $p \equiv 1 \pmod{6}$ with $13 \leq p \leq 400$,
we enumerate and cache the integers which are sums of two sixth powers modulo $p$. Now, we test integers less than or equal to $164634913$ in each of the eight residue
classes modulo $504 = 7 \cdot 8 \cdot 9$. We remove values of $k$
that are not sixth-power free, that are divisible by a prime $\equiv 3 \pmod{4}$, that are sums of two integer sixth powers, and that reduce modulo some $p \equiv 1 \pmod{6}$ to an element of $\F_{p}$ that is not a sum of two sixth powers. The result is a list of $111625$ values of $k < 164634913$ which are not sums of two integer sixth powers and for which $C_{k}$ is locally solvable. The computation runs in 46.29 seconds and the code can be found in the script \href{https://github.com/newtan18/Sums-of-Two-Sixth-Powers/blob/main/step1-localtest.txt}{\tt step1-localtest.txt}.

\section{Maps from $C_{k}$ to elliptic curves}
\label{maps}

The curve $C_{k} : x^{6} + y^{6} = kz^{6}$ has at least $72$ automorphisms
defined over $\Q(\zeta_{6})$, generated by the maps $\mu_{1}(x : y : z) = (\zeta_{6} x : y : z)$, $\mu_{2}(x : y : z) = (x : \zeta_{6} y : z)$ and $\mu_{3}(x : y : z) = (y : x : z)$. Magma can work out the action of each of these maps on the $10$-dimensional space of holomorphic $1$-forms on $C_{k}$. We find
$8$ subgroups $H$ of $\Aut(C_{k})$ for which the quotient curve $C_{k}/H$ has
genus $1$ and the corresponding one-dimensional subspaces of holomorphic
$1$-forms are distinct. From these it is not difficult to compute the
corresponding map to an elliptic curve.

For example, one such subgroup is $\langle \mu_{1}^{5} \mu_{2}, \mu_{2}^{3} \mu_{3} \rangle$. The monomials $x^{3} y^{3}$, $xyz^{4}$ and $x^{6} - y^{6}$ are all
fixed by $\mu_{1}^{5} \mu_{2}$ and each are sent to their negative by
$\mu_{2}^{3} \mu_{3}$. If $\phi : C_{k} \to \P^{2}$ is given by
$\phi((x : y : z)) = (x^{3} y^{3} : xyz^{4} : x^{6} - y^{6})$, then we have
$\phi(P) = \phi(\alpha(P))$ for all points $P$ on $C_{k}$ and all $\alpha \in \langle \mu_{1}^{5} \mu_{2}, \mu_{2}^{3} \mu_{3} \rangle$. Letting $a = x^{3}y^{3}$, $b = xyz^{4}$ and $c = x^{6} - y^{6}$, the image of $\phi$ is the curve
\[
  a^{3} - \frac{k^{2}}{4} b^{3} + \frac{1}{4} ac^{2} = 0.
\]
This curve has genus $1$, and thus is the quotient curve $C_{k}/\langle \mu_{1}^{5} \mu_{2}, \mu_{2}^{3} \mu_{3} \rangle$. This curve has the point
$(0 : 0 : 1)$ on it, and a change of variables turns this into the
elliptic curve $E_{-4k^{4}}$. Composing these maps gives the map
$\phi : C_{k} \to E_{-4k^{4}}$ given by $\phi(x : y : z) = (k^{2} xyz^{4} : -k^{2} x^{6} + k^{2} y^{6} : x^{3} y^{3})$.

The table below lists all 10 independent maps from $C_{k}$ to elliptic curves.

\begin{center}
\begin{tabular}{ccc}
Subgroup of $\Aut(C_{k})$ & Codomain & Map\\
\hline
$\langle \mu_{1}^{2} \mu_{2}^{3} \rangle$ & $E_{k}$ & $(x,y) \mapsto (-y^{2}, x^{3})$\\
$\langle \mu_{1}^{3} \mu_{2}^{2} \rangle$ & $E_{k}$ & $(x,y) \mapsto (-x^{2}, y^{3})$\\
$\langle \mu_{1} \mu_{2}^{2} \rangle$ & $E_{4k}$ & $(x,y) \mapsto \left(\frac{x^{4}}{y^{2}},\frac{x^{6}+2y^{6}}{y^{3}}\right)$\\
$\langle \mu_{1}^{2} \mu_{2} \rangle$ & $E_{4k}$ & $(x,y) \mapsto \left(\frac{y^{4}}{x^{2}},\frac{2x^{6}+y^{6}}{x^{3}}\right)$\\
$\langle \mu_{1} \mu_{2}^{3} \rangle$ & $E_{-k^{2}}$ & $(x,y) \mapsto \left(\frac{k}{y^{2}}, \frac{kx^{3}}{y^{3}}\right)$\\
$\langle \mu_{1}^{3} \mu_{2} \rangle$ & $E_{-k^{2}}$ & $(x,y) \mapsto \left(\frac{k}{x^{2}}, \frac{ky^{3}}{x^{3}}\right)$\\
$\langle \mu_{1} \mu_{2}^{5}, \mu_{1}^{2} \mu_{2} \rangle$ & $E_{16k^{2}}$ & $(x,y) \mapsto \left(-4x^{2}y^{2}, -8x^{6} + 4k\right)$\\
$\langle \mu_{1} \mu_{2}^{4} \rangle$ & $E_{k^{3}}$ & $(x,y) \mapsto \left(\frac{kx^{2}}{y^{2}}, \frac{k^{2}}{y^{3}}\right)$\\
$\langle \mu_{1}^{4} \mu_{2} \rangle$ & $E_{k^{3}}$ & $(x,y) \mapsto \left(\frac{ky^{2}}{x^{2}}, \frac{k^{2}}{x^{3}}\right)$\\
$\langle \mu_{1}^{5} \mu_{2}, \mu_{2}^{3} \mu_{3} \rangle$ & $E_{-4k^{4}}$ & $(x,y) \mapsto \left(\frac{k^{2}}{x^{2} y^{2}}, \frac{-k^{2} x^{6} + k^{2}y^{6}}{x^{3}y^{3}}\right)$
\end{tabular}
\end{center}

We wish to note that for the maps from $C_{k} \to E_{4k}$, the quotient curve by
the subgroup indicated (either $\langle \mu_{1} \mu_{2}^{2} \rangle$ or $\langle \mu_{1}^{2} \mu_{2} \rangle$) is the genus two hyperelliptic curve given by $D_{k} : y^{2} = \frac{1}{k} x^{6} + \frac{1}{4k^{2}}$.
This equation may be rewritten as
\[
  \left(\frac{2ky}{x^{3}}\right)^{2} = \left(\frac{1}{x^{2}}\right)^{3} + 4k
\]
As a consequence, we have the map $\phi(x,y) = \left(\frac{1}{x^{2}},\frac{2ky}{x^{3}}\right)$ from $D_{k} \to E_{4k}$. (The authors did not find a subgroup of
$\Aut(C_{k})$ that fixed a one-dimensional space of differentials corresponding
to these maps.)

\begin{thm}
Suppose that $k$ is a sixth-power free integer and $P = (x,y)$ is a rational point on $C_{k}$ and the image of $P$ under one of the ten maps given above is a torsion point. Then $k = 1$ or $k = 2$.
\end{thm}
\begin{proof}
  Apart from the cases of $E_{a^{6}}$ and $E_{-432a^{6}}$, every torsion point on $E_{a}$ has the $x$ or $y$ coordinate zero. Inspecting the ten maps above, we find that if $P \in C_{k}(\Q)$ and its image on $E_{a}$ has the $x$ or $y$ coordinate zero, then $x = 0$ or $y = 0$ or (for the seventh or tenth maps) that $x^{6} = y^{6} = k/2$. This implies that $k/2$ is a sixth power, but since $k$ is sixth-power free, $k = 2$.

Now, we consider the cases that $E_{a} = E_{\alpha^{6}}$ or $E_{a} = E_{-432 \alpha^{6}}$ for $a \in \{k$, $4k$, $-k^{2}$, $16k^{2}$, $k^{3}$, $-4k^{4} \}$. 
If $\alpha$ is a rational number and $k = \alpha^{6}$ is a sixth
power, this forces $k = 1$. If $4k = \alpha^{6}$ and $k$ is
sixth-power free, then $k = 16$ but $x^{6} + y^{6} = 16$ has no points
in $\Q_{2}$. The cases $-k^{2} = -432\alpha^{6}$ and $-4k^{4} = -432\alpha^{6}$
never occur. If $16k^{2} = \alpha^{6}$, then $k = 2$. Finally, if $k^{3} = \alpha^{6}$, then $k$ is a perfect square. In this case,
$E_{k^{3}}$ has the torsion points $(2 \alpha^{2}, \pm 3 \alpha^{3})$. However,
we have that $2 \alpha^{2} = \frac{\alpha^{2} x^{2}}{y^{2}}$ or $\frac{\alpha^{2} y^{2}}{x^{2}}$, which implies that
$2 = \frac{x^{2}}{y^{2}}$ or $\frac{y^{2}}{x^{2}}$, contradicting the irrationality of $\sqrt{2}$. 
\end{proof}

As a consequence of the above result, if $k \not\in \{ 1, 2 \}$ is sixth-power free and the rank of one of the six elliptic curves $E_{k}$, $E_{4k}$, $E_{-k^{2}}$, $E_{16k^{2}}$, $E_{k^{3}}$ or $E_{-4k^{4}}$ is zero, then $k$ is not a sum of
two rational sixth powers. For each of the $111625$ values of $k$ found in the previous section, we need to determine the Mordell-Weil group (or a finite index subgroup thereof) of one of these six curves. The most straightforward
approach to this problem is to conduct a $2$-descent. However, a $2$-descent
on $E_{k}$ requires computing the class group of $\Q(\sqrt[3]{-k})$,
and this is time-consuming to do unconditionally if $k$ is large. We proceed
to apply a number of other techniques specific to our situation and resort
to an unconditional $2$-descent only when absolutely necessary.

\section{Checking if $L(E_{k^{3}},1) = 0$}

The elliptic curve $E_{k^{3}} : y^{2} = x^{3} + k^{3}$ is a quadratic twist
of $E_{1} : y^{2} = x^{3} + 1$. If $k \equiv 1 \pmod{8}$, the sign of the
functional equation for $E_{k^{3}}$ is $1$, while if $k \equiv 2 \pmod{8}$,
the sign of the functional equation is $-1$. We are able to rule out
most odd values of $k$ by showing that $L(E_{k^{3}},1) \ne 0$.

Waldspurger's theorem \cite{Waldspurger} says, very roughly speaking, that
\[
  \sum_{k} k^{1/4} \sqrt{L(E_{k^{3}},1)} q^{k}
\]
is a weight $3/2$ modular form of a particular level. In Theorem~11 of \cite{Purkait}, Soma Purkait works out the predictions of Waldspurger's theorem, showing
that there is a modular form $f = \sum b(k) q^{k}$ of level $576$ and trivial character whose Fourier coefficients encode the $L$-values of $L(E_{k^{3}},1)$ under the assumption that
$3 \nmid k$. In Example 2 of \cite{Purkait}, Purkait gives a complicated
formula for this modular form $f$ in terms of ternary theta series. We are able
to find a formula more amenable to computation using the theta series for
the six ternary quadratic forms
\begin{align*}
  Q_{1} &= x^{2} + 4y^{2} + 144z^{2},\\
  Q_{2} &= 4x^{2} - 4xy + 5y^{2} + 36z^{2},\\
  Q_{3} &= 4x^{2} + 9y^{2} + 16z^{2},\\
  Q_{4} &= x^{2} + 16y^{2} + 36z^{2},\\
  Q_{5} &= 4x^{2} + 13y^{2} + 10yz + 13z^{2}, \text{ and }\\
  Q_{6} &= 4x^{2} + 4y^{2} + 4yz + 37z^{2}.
\end{align*}
These six quadratic forms constitute a single genus. Let
\[
  h = \frac{5}{16} \theta_{Q_{1}} - \frac{3}{16} \theta_{Q_{2}} - \frac{7}{16} \theta_{Q_{3}} + \frac{5}{16} \theta_{Q_{4}} + \frac{9}{16} \theta_{Q_{5}} - \frac{3}{16} \theta_{Q_{6}} = \sum c(n) q^{n}.
\]
Then for $k \equiv 1 \pmod{24}$ we have $c(k) = b(k)$, and
if $k \equiv 17 \pmod{24}$, we have $c(k) = 6b(k)$. It follows that if $k \equiv 1 \pmod{8}$ is a fundamental discriminant and $c(k) \ne 0$, then
$L(E_{k^{3}},1) \ne 0$. Hence $E_{k^{3}}$ has rank zero, and if $k > 1$ this implies that $k$ is not the sum of two rational sixth powers. (If $k$ is not squarefree, we can simply replace $k$ with $k/m^{2}$ in the above calculation.)

Each theta series above can be computed by multiplying a binary theta series by a unary theta series. In this way, it is possible to compute the first $165$ million coefficients of $h$ and among these determine the odd values of $k$
for which $L(E_{k^{3}},1) \ne 0$. Of the $111625$ values of $k$ for which $C_{k}$ is locally solvable, $55284$ are odd while $56341$ are even. The
computation just described rules out all but $2753$ odd values of $k$. The computation takes 559.20 seconds, and the code run can be found in the script \href{https://github.com/newtan18/Sums-of-Two-Sixth-Powers/blob/main/step2-wald.txt}{\tt step2-wald.txt}.

\section{Computing Mordell-Weil groups}

Here and elsewhere, we rely on the procedure for explicit $n$-descent
developed by Cremona, Fisher, O'Neil, Simon and Stoll in
\cite{Descent1}, \cite{Descent2}, \cite{Descent3} and implemented in
Magma with much of the code written by Michael Stoll, Tom Fisher, and
Steve Donnolly.

First, we use that each elliptic curve $E_{a}$ has a cyclic $3$-isogeny. We take the remaining $59094$ values of $k$ and compute the $3$-isogeny Selmer groups to bound the rank for the elliptic curves in the set $\{ E_{k}, E_{4k}, E_{-k^{2}}, E_{16k^{2}}, E_{-4k^{4}} \}$. We hope to rule out $k$'s for which one of these curves has rank zero and so we only test elliptic curves with
root number equal to $1$. This test is run in \href{https://github.com/newtan18/Sums-of-Two-Sixth-Powers/blob/main/step3-3isog.txt}{\tt step3-3isog.txt} and takes
a bit under 6 hours (namely 20551.4 seconds). This step rules out $39586$ values of $k$, and $19508$ values remain.

Second, for each of the $19508$ remaining $k$'s, we perform a full $3$-descent
by doing a first and second $3$-isogeny descent (via the Magma command {\tt ThreeDescentByIsogeny}) on $E_{k}$, $E_{4k}$ and $E_{k^{3}}$. For these curves, this command requires class group computations of low discriminant quadratic and cubic fields. We search for points on the resulting $3$-covers in the hope that we can provably compute the rank of $E_{k}$, $E_{4k}$ or $E_{k^{3}}$. This test is run in \href{https://github.com/newtan18/Sums-of-Two-Sixth-Powers/blob/main/step4-findMW.txt}{\tt step4-findMW.txt}. Once the $3$-covers are found, we sort the curves in increasing order of the upper bound on the rank and search for points on the associated $3$-covers with a height bound of $10000$. If we are not successful, we double the height bound and search again. If we are not successful at a height bound of $320000$, we give up. This is the most time consuming step of the process, taking about 26 hours. This step finds $864$ additional values of $k$ for which one of $E_{k}$, $E_{4k}$ or $E_{k^{3}}$ has rank zero. In the end we succeed
in computing the Mordell-Weil groups of one of $E_{k}$, $E_{4k}$, or $E_{k^{3}}$
for all but $196$ values of $k$. There are also 34 cases where the elliptic curve in question has rank 5, and one case ($k = 123975217$) for which the rank
of $E_{k}$ is 6. For these $k$'s we seek to find the Mordell-Weil group of a different elliptic curve.

Third, for each of the $196+35 = 231$ remaining $k$'s, we obtain as much unconditional information as possible about the ranks of the six elliptic curves using descent by $3$-isogeny, as well as a $2$-descent on $E_{k^{3}}$ (combined with the Cassels-Tate pairing to identify $2$-covers as corresponding to an element of the Shafarevich-Tate group of $E_{k^{3}}$). Once these unconditional upper bounds on ranks have been obtained, we search for points on these curves by assuming GRH and performing $2$-descents and $4$-descents on all six curves and searching for points on the $2$-covers and $4$-covers to see if enough independent points are found to match the unconditional rank upper bound. This takes about 22 minutes (1292.52 seconds). The code that runs these computations is available
in the scripts \href{https://github.com/newtan18/Sums-of-Two-Sixth-Powers/blob/main/step5-24descent.txt}{\tt step5-24descent.txt} and \href{https://github.com/newtan18/Sums-of-Two-Sixth-Powers/blob/main/step5-highrank.txt}{\tt step5-highrank.txt}. Of the
$196$ $k$'s for which generators were not found, this step is unsuccessful
for $26$, and of the $35$ $k$'s for which one of $E_{k}$, $E_{4k}$ or $E_{k^{3}}$ has rank $5$ or $6$, this is unsuccessful for $4$ values of $k$.

Fourth, for each of the 30 remaining values of $k$ we use the method
of Tom Fisher \cite{Fisher12} to search for points using $12$-descent.
For each of the remaining $k$'s there is at least one elliptic curve $E_{a}$
for which we have an unconditional upper bound on the rank of $1$, and
for which the root number is $-1$. For each such $k$, we choose $a$ minimal
subject to these conditions and perform a conditional $12$-descent
and search for points. We succeed in finding a generator in $23$ cases.
We fail to find a generator for the following seven values of $k$:
$49897450$, $117092530$, $120813050$, $128327978$, $130187450$, $149477050$,
and $160631290$. (For $k = 128327978$, $E_{k^{3}}$ has rank $5$ with easily
found generators.) This is performed with the script \href{https://github.com/newtan18/Sums-of-Two-Sixth-Powers/blob/main/step6-12desc.txt}{\tt step6-12desc.txt}
and the running time is just over 3 hours (11180.99 seconds).

So far, we have avoided doing an unconditional $2$-descent on any
elliptic curve other than $E_{k^{3}}$ because of the cost of computing
the class group of a (potentially high discriminant) cubic field. We
now do this for the remaining seven values of $k$. For each $k$, we
choose the elliptic curve for which the corresponding Minkowski bound
is the smallest. For $k = 49897450$ and $k = 149477050$, this shows
that $E_{k}$ has rank zero. For $k = 120813050$ and $k = 130187450$
this shows that $E_{4k}$ has rank zero. For $k = 128327978$, the
computation shows that $E_{16k^{2}}$ has rank zero (using both a
$2$-descent and the Cassels-Tate pairing). For $k = 117092530$ and $k
= 160631290$, a $2$-descent shows that $E_{k}$ has rank $1$ (while
previously our unconditional bound on the rank had been $3$). This is
performed with the script \href{https://github.com/newtan18/Sums-of-Two-Sixth-Powers/blob/main/step7-2descents.txt}{\tt step7-2descents.txt}, and the running
time is just under 2 hours (7153.69 seconds). For $k = 117092530$, the
Minkowski bound for $\Q(\sqrt[3]{-k})$ is $57383551$, and the time needed
for the proof phase of the class group computation is $3327.08$ seconds.

Of the $19508$ values of $k$ that remained after step $3$, $864$ were
removed in step $4$ and five more were removed in step $7$. For each of the
remaining $18639$ values of $k$, we know a finite-index subgroup of
the Mordell-Weil group of one of the six corresponding elliptic
curves, and moreover that curve has rank less than or equal to $4$. In fact,
of the $18639$ values of $k$, the chosen elliptic curve has rank $1$
in $16032$ cases, rank $2$ in $1172$ cases, rank $3$ in $1371$ cases,
and rank $4$ in only $64$ cases.

\section{Using the Mordell-Weil sieve}

As indicated in Section~\ref{back}, the goal of the Mordell-Weil sieve
is to choose an integer $N$ and a finite set $S$ of primes $p$ of good
reduction for $E$ and consider the diagram
\[
\xymatrix{
C_{k}(\Q) \ar[rr]^{\iota} \ar[d] & & E(\Q)/NE(\Q) \ar[d]^{\alpha}\\
\prod_{p \in S} C_{k}(\F_{p}) \ar[rr]^{\beta} & & \prod_{p \in S} E(\F_{p})/NE(\F_{p})}.
\]
If one finds that $\im \alpha \cap \im \beta = \emptyset$, then $C_{k}(\Q)$ must
be empty. Here $E$ can be any of the six elliptic curves $E_{k}$, $E_{4k}$, $E_{-k^{2}}$, $E_{16k^{2}}$, $E_{k^{3}}$, or $E_{-4k^{4}}$. In practice, we are not always
able to provably find $E(\Q)$. Instead, we have used the {\tt Saturation} command in Magma to compute a finite index subgroup
$A \subseteq E(\Q)$ with the property that $[E(\Q) : A]$ is not divisible
by any primes $p \leq 100$. It follows that if there is no prime $\ell > 100$
for which $\ell \mid N$, then $A/NA \cong E(\Q)/NE(\Q)$, and we may use
$A$ in place of $E(\Q)$ in the diagram above. In practice, the largest $N$
we need to use is $N = 84$.

Before discussing the method and results, we begin with a simple
example. Let $k = 138826$. We have $E_{4k}(\Q) \cong \Z$, and a generator is
\[
  P = \left(\frac{605879737}{2358^{2}}, \frac{-17828809046227}{2358^{3}}\right).
\]
We use the map $\phi : C_{k} \to E_{4k}$ given by $\phi(x,y) =
\left(\frac{x^{4}}{y^{2}}, \frac{x^{6} + 2y^{6}}{y^{3}}\right)$.  We
find that $C_{k}(\F_{5})$ contains $6$ points, $E_{4k}(\F_{5}) \cong \Z/6\Z$, but that the image of $C_{k}(\F_{5}) \to E_{4k}(\F_{5})$ consists of $3$ points.  The
reduction $\tilde{P} \in E_{4k}(\F_{5})$ has order $6$, and if $n$ is
an integer, then $nP$ reduces to a point in $E_{4k}(\F_{5})$ that is
in the image of $C_{k}(\F_{5}) \to E_{4k}(\F_{5})$ if and only if $n$
is even. It follows that if $Q \in C_{k}(\Q)$, then $\phi(Q) = nP$ for
some even $n$.

Now we consider reduction modulo $7$. In this case $C_{k}(\F_{7})$ has $36$
points and the image of $C_{k}(\F_{7}) \to E_{4k}(\F_{7})$ consists of $6$ points.
We have $E_{4k}(\F_{7}) \cong \Z/2\Z \times \Z/6\Z$ and the reduction
$\tilde{P} \in E_{4k}(\F_{7})$ again has order $6$. This, time we
find that $nP$ reduces to a point in $E_{4k}(\F_{7})$ that is in
the image of $C_{k}(\F_{7}) \to E_{4k}(\F_{7})$ if and only if $n \equiv 1 \text{ or } 5 \pmod{6}$. It follows that if $Q \in C_{k}(\Q)$ then $\phi(Q) = nP$
for some odd $n$, and this contradicts the previous paragraph. Thus $C_{k}(\Q) = \emptyset$. 

As explained in Section 3.2 of \cite{BruinStoll}, the sets $A/NA$ can
be very large if $N$ is large or if the rank of $E$ is high. For this
reason we follow their suggestion of successively raising $N$ one
prime factor at a time. Suppose that we have already computed the
admissible elements of $A/NA$ (i.e. those that could possibly occur as
the image of a point from $C_{k}(\Q)$) by sieving using a collection
of small primes $S$. We then choose a small prime $r$ and set $N' =
rN$. Then we find the full preimage of the admissible elements in
$A/N'A$, retest their admissibility for primes in $S$, and possibly
test a further set of primes.  Unlike the case of \cite{BruinStoll},
the maximum $N$ needed to prove that $C_{k}(\Q)$ is empty is never
more than $84$ (while Bruin and Stoll report occasionally needing to
have $N$ as large as $10^{100}$).

As an example, consider the case of $k = 3506050$. The elliptic curve
$E_{k}$ has rank $4$ and trivial torsion subgroup. First, we let $N = 2$ and test the primes $p$ of good reduction less than or equal to
$311$.  We find that of the 16 elements of $A/2A$,
$9$ are admissible. We then increase $N$ to $4$ and begin with $9
\cdot 16 = 144$ elements of $A/4A$. We retest
their admissibility for primes less than or equal to $311$ and find
that all of them are admissible. We then increase $N$ from $4$ to $12$
and test primes $p \leq 479$. Initially, we had $11664$ elements of
$A/12A$, but this is reduced to $1296$. Next, we
increase $N$ from $12$ to $84$ and start with $3111696$ elements of
$A/84A$. Testing for $p \leq 229$ reduces this to $1204$ elements,
and by the time we test $p = 1021$, no admissible elements remain. Hence
$C_{3506050}(\Q) = \emptyset$. The total time required for this $k$ was
$508$ seconds, and this is the most time consuming of all the $k$'s we test.

Compared to the previous steps, the Mordell-Weil
sieve step is comparatively fast, taking about 35 minutes (2107.69 seconds)
to show that $C_{k}(\Q) = \emptyset$ for all $18639$ remaining $k$'s with
$k < 164634913$. This computation is performed by the script \href{https://github.com/newtan18/Sums-of-Two-Sixth-Powers/blob/main/step8-MWsieve.txt}{\tt step8-MWsieve.txt}. This concludes the proof of Theorem~\ref{main}. Below is
a table summarizing the steps in the computation and the time required for each.

\begin{center}
\begin{tabular}{c|c|c|c}
Step & Task & Run time (seconds) & $k$'s eliminated\\
\hline  
1 & Local solvability & 46.29 & $164523287$\\
2 & $L(E_{k^{3}},1) \ne 0$ & 559.20 & $52531$\\
3 & 3-isogeny descent & 20551.4 & $39586$\\
4 & Full 3-descent & 119076 & $864$\\
5 & Conditional descent & 1292.52 & $0$\\
6 & 12-descent & 11180.99 & $0$\\
7 & Unconditional 2-descent & 7153.69 & $5$\\
8 & Mordell-Weil sieve & 2107.69 & $18639$\\
\hline
Total & & $161968$ & $164634912$\\
\end{tabular}
\end{center}

\section{Concluding remarks}

As mentioned in the introduction, it is natural to consider the
problem of finding the smallest positive integer which is a sum of two
rational $n$th powers but not a sum of two integer $n$th powers. If $n
= 5$, the curve $D_{k} : x^{5} + y^{5} = k$ admits no map to an
elliptic curve, and the projective closure of $D_{k}$ always has a
rational point (namely $(-1 : 1 : 0)$). This precludes the possibility
of ruling out rational points on $D_{k}$ using local methods or
the Mordell-Weil sieve. For these reasons, the $n = 5$ case appears to be more
challenging than the $n = 4$ or $n = 6$ cases.

Similar techniques should allow one to approach the cases of $n = 8$ and $n = 12$ where there are maps from $x^{n} + y^{n} = k$ to elliptic curves,
but the smallest values of $k$ for which these curves are known to have rational non-integer points are $\left(\frac{50429}{17}\right)^{8} + \left(\frac{43975}{17}\right)^{8} \approx 8 \cdot 10^{27}$ and
$\left(\frac{9298423}{17}\right)^{12} + \left(\frac{8189146}{17}\right)^{12} \approx 7.6 \cdot 10^{46}$, respectively. The size of these numbers would make an exhaustive search prohibitively time-consuming.

\bibliographystyle{plain}
\bibliography{refs}

\end{document}